\documentclass[11pt]{amsart}
\usepackage{amsmath,amsthm,amssymb,latexsym,url}
\usepackage[margin=1in]{geometry}
\newcommand{\EMDE}{Erd\H{o}s-Moser Diophantine equation}
\newcommand{\osn}[1]{\oldstylenums{#1}}

\theoremstyle{plain}
\newtheorem{thm}{Theorem}
\newtheorem{lem}{Lemma}
\newtheorem{cnj}{Conjecture}
\newtheorem{prop}{Proposition}

\title[Divisibility of Power Sums and the Generalized Erd\H{os}-Moser Equation]{Divisibility of Power Sums and\\ the Generalized Erd\H{os}-Moser Equation}
\author{Kieren MacMillan and Jonathan Sondow}
\date{}

\begin{document}
\baselineskip=1.4em
\begin{abstract}
Using elementary methods, we determine the highest power of $2$ that divides a power sum $1^n + 2^n + \dotsb + m^n$\!, generalizing Lengyel's formula for the special case where $m$ is itself a power of $2$. An application is a simple proof of Moree's result that, in any solution of the generalized Erd\H{o}s-Moser Diophantine equation
\begin{align*}
	1^n + 2^n + \dotsb + (m-1)^n = am^n\!,
\end{align*}
$m$ is odd.
\end{abstract}
\maketitle

%%%  SEC: INTRODUCTION
\section{Introduction} \label{SEC: introd}

For $p$ a prime and $k$ an integer, \emph{$v_p(k)$ denotes the highest exponent $v$ such that $p^v\mid k$}. (Here $a\mid b$ means $a$ divides $b$.) For example, $v_2(k)=0$ if and only if $k$ is odd, and $v_2(40)=3$.

For any \emph{power sum}
\begin{equation*}
	S_n(m) := \sum_{j=1}^m j^n =1^n + 2^n + \dotsb + m^n \qquad (m>0,\, n>0), %\label{EQ: power sum}
\end{equation*}
we determine $v_2\!\left(S_n(m)\right)$. As motivation, we first give a classical extension of the fact that $S_1(m) = m(m+1)/2$, a formula known to the ancient Greeks~\cite[Ch.~1]{Dickson} and famously \cite{hayes} derived by Gauss at age seven to calculate the sum
\begin{equation*}
	1 + 2 + \dotsb + 99 + 100 = (1+100) + (2+99) + \cdots + (50+51) = 5050.
\end{equation*}

\begin{prop} \label{THM: m(m+1)/2 | S_n(m)}
If $n>0$ is odd and $m>0$, then $m(m+1)/2$ divides $S_n(m)$.
\end{prop}

The proof is a modification of Lengyel's arguments in \cite{Lengyel} and \cite{Lengyel2}.

\begin{proof}[Proof of Proposition \ref{THM: m(m+1)/2 | S_n(m)}]
\emph{Case 1: both $n$ and $m$ odd}. Since $m$ is odd we may group the terms of $S_n(m)$ as follows, and as $n$ is also odd we see by expanding the binomial that
\begin{align*}
	S_n(m) = m^n + \sum_{j=1}^{(m-1)/2} \!\left(j^n + (m- j)^n\right)  \quad \Longrightarrow  \quad m \mid S_n(m).
\end{align*}
Similarly, grouping the terms in another way shows that
\begin{align*}
	S_n(m) = \frac12 \sum_{j=1}^m \! \left(j^n + \left((m+1)- j\right)^n\right)
		\quad \Longrightarrow \quad \frac{m+1}{2} \mid S_n(m).
\end{align*}
As $m$ and $m+1$ are relatively prime, it follows that $m(m+1)/2\mid S_n(m)$.\\
\emph{Case 2: $n$ odd and $m$ even}. Here
\begin{align*}
S_n(m) = \sum_{j=1}^{m/2} \!\left(j^n + ((m+1)- j)^n\right)  \quad \Longrightarrow  \quad (m+1) \mid S_n(m)
\end{align*}
and
\begin{align*}
	S_n(m) = \frac12\sum_{j=0}^m \! \left(j^n + \left(m- j\right)^n\right)
		\quad \Longrightarrow \quad \frac{m}{2} \mid S_n(m).
\end{align*}
Thus $m(m+1)/2\mid S_n(m)$ in this case, too.
\end{proof}

Here is a paraphrase of Lengyel's comments \cite{Lengyel} on Proposition \ref{THM: m(m+1)/2 | S_n(m)}:

\def\changemargin#1#2{\list{}{\rightmargin#2\leftmargin#1}\item[]}
\let\endchangemargin=\endlist 
\begin{changemargin}{0in}{0in}
\begin{quotation}
\noindent We note that Faulhaber had already known in 1631 (cf. \cite{Edwards}) that $S_n(m)$ can be expressed as a polynomial in $S_1(m)$ and $S_2(m)$, although with fractional coefficients. In fact, $S_n(m)/(2m+1)$ or $S_n(m)$ can be written as a polynomial in $m(m + 1)$ or $\left(m(m + 1)\right)^2$, if $n$ is even or $n\ge3$ is odd, respectively.
\end{quotation}
\end{changemargin}

Proposition \ref{THM: m(m+1)/2 | S_n(m)} implies that if $n$ is odd, then
\begin{equation*}
	v_p(S_n(m)) \ge v_p(m(m+1)/2),
\end{equation*}
for any prime $p$. When $p=2$, Theorem~\ref{THM: Sigma(2^d)} shows that the inequality is strict for odd $n>1$.

\begin{thm} \label{THM: Sigma(2^d)}
Given any positive integers $m$ and $n$, the following divisibility formula holds:
\begin{equation} \label{EQ: main}
\begin{aligned}
	v_2\!\left(S_n(m)\right) =
		\begin{cases}
			\, v_2(m(m+1)/2) &\text{if $n=1$ or $n$ is even,}  \\[0.4em]
			\, 2v_2(m(m+1)/2) &\text{if $n\ge3$ is odd.}
		\end{cases}
\end{aligned} 
\end{equation} 
\end{thm}

The elementary proof given in Section~\ref{SEC: Proof THM: Sigma(2^d)} uses a lemma proved by induction.

In the special case where $m$ is a power of $2$, formula \eqref{EQ: main} is due to Lengyel \cite[Theorem~1]{Lengyel}. His complicated proof, which uses Stirling numbers of the second kind and von Staudt's theorem on Bernoulli numbers, is designed to be generalized. Indeed, for $m$ a power of an odd prime $p$, Lengyel proves results towards a formula for $v_p\!\left(S_n(m)\right)$ in \cite[Theorems~3,\,4,\,5]{Lengyel}.

In the next section, we apply formula \eqref{EQ: main} to a certain Diophantine equation.

%%%  SEC: EME & GEME
\section{Equations of Erd\H{o}s-Moser Type} \label{SEC: EME&GEME}

As an application of Theorem~\ref{THM: Sigma(2^d)}, we give a simple proof of a special case of a result due to Moree. Before stating it, we discuss a conjecture made by Erd\H{o}s and Moser \cite{Moser} around \osn{1953}.

\begin{cnj}[Erd\H{os}-Moser] \label{CNJ: EMC}
The only solution of the Diophantine equation
\begin{equation*}
	1^n + 2^n + \dotsb + (m-1)^n = m^n
\end{equation*}
is the trivial solution \mbox{$1 + 2 = 3$}.
\end{cnj}

Moser proved, among many other things, that \emph{Conjecture \ref{CNJ: EMC} is true for odd exponents $n$}. (An alternate proof is given in \cite[Corollary~1]{MS}.) In \osn{1987} Schinzel showed that \emph{in any solution, $m$ is odd}~\cite[p.~800]{MoreeEtAl}. For surveys of results on the problem, see \cite[Section~D7]{Guy}, \cite{Moree}, \cite{MoreeOnMoser}, and \cite{MoreeEtAl}.

In \osn{1996} Moree generalized Conjecture \ref{CNJ: EMC}.
\begin{cnj}[Moree] \label{CNJ: GEMC}
The only solution of the generalized \EMDE
\begin{equation}
	1^n + 2^n + \dotsb + (m-1)^n = am^n  \label{EQ: GEMC}
\end{equation}
is the trivial solution $1 + 2 + \dotsb + 2a = a(2a+1)$.
\end{cnj}

Actually, Moree~\cite[p.~290]{Moree} conjectured that \emph{equation~\eqref{EQ: GEMC} has no integer solution with $n>1$}. The equivalence to Conjecture \ref{CNJ: GEMC} follows from the formula
\begin{equation}
	1 + 2 + \dotsb + k= \frac12 k(k+1) \label{EQ: m(m-1)/2}
\end{equation}
with $k=m-1$.

Generalizing Moser's result on Conjecture \ref{CNJ: EMC}, Moree~\cite[Proposition~3]{Moree} proved that \emph{Conjecture \ref{CNJ: GEMC} is true for odd exponents $n$}. He also proved a generalization of Schinzel's result.

\begin{prop}[Moree] \label{THM: GEMC}
If equation \eqref{EQ: GEMC} holds, then $m$ is odd.
\end{prop}

In fact, Moree~\cite[Proposition~9]{Moree} (see also \cite{MoreeOnMoser}) showed more generally that \emph{if \eqref{EQ: GEMC} holds and a~prime~$p$ divides~$m$, then $p-1$ does not divide~$n$}. (The case $p=2$ is Proposition~\ref{THM: GEMC}.) His proof uses a congruence which he says \cite[p.~283]{Moree} can be derived from either the von Staudt-Clausen theorem, the theory of finite differences, or the theory of primitive roots.

We apply Theorem~\ref{THM: Sigma(2^d)} to give an elementary proof of Proposition~\ref{THM: GEMC}.

\begin{proof}[Proof of Proposition~\ref{THM: GEMC}.]
If $n=1$, then \eqref{EQ: GEMC} and \eqref{EQ: m(m-1)/2} show that $m=2a+1$ is odd.

If $n>1$ and $m$ is even, set $d:=v_2(m)=v_2(m(m+1))$. Theorem~\ref{THM: Sigma(2^d)} implies $v_2(S_n(m)) \le 2(d-1)$, and~\eqref{EQ: GEMC} yields $S_n(m) = S_n(m-1) + m^n = (a+1)m^n.$ But then $nd \le v_2\!\left(S_n(m)\right) \le 2(d-1)$, contradicting $n>1$. Hence $m$ is odd.
\end{proof}

%%%  SEC: PROOF OF THEOREM 1
\section{Proof of Theorem~\ref{THM: Sigma(2^d)}} \label{SEC: Proof THM: Sigma(2^d)}

The heart of the proof of the divisibility formula is the following lemma.

\begin{lem}
Given any positive integers $n,d,$ and $q$ with $q$ odd, we have
\begin{equation}
\begin{aligned}
	v_2\!\left(S_n(2^dq)\right) =
		\begin{cases}
			\, d-1 &\text{if $n=1$ or $n$ is even,}  \\[0.4em]
			\, 2(d-1) &\text{if $n\ge3$ is odd.}
		\end{cases} \label{EQ: m=2^dq}
\end{aligned} 
\end{equation}
\end{lem}

\begin{proof}
We induct on $d$. Since the power sum for $S_n(2q)$ has exactly $q$ odd terms, we have $v_2(S_n(2q)) = 0$, and so \eqref{EQ: m=2^dq} holds for $d=1$. By \eqref{EQ: m(m-1)/2} with $k=2^dq$, it also holds for all $d\ge1$ when $n=1$. Now assume inductively that \eqref{EQ: m=2^dq} is true for fixed $d \ge 1$.

Given a positive integer $a$, we can write the power sum $S_n(2a)$ as
\begin{align*}
	S_n(2a) = a^n + \sum_{j=1}^a \left((a - j)^n + (a + j)^n\right)
		&= a^n + 2 \sum_{j=1}^a \sum_{i=0}^{\lfloor n/2 \rfloor} \binom{n}{2i} a^{n-2i} j^{2i}\\
	         &= a^n + 2\hspace{-0.25em}\sum_{i=0}^{\lfloor n/2 \rfloor}
		\binom{n}{2i} a^{n-2i}S_{2i}(a).
\end{align*}

If $n \ge 2$ is even, we extract the last term of the summation, set $a=2^dq$, and write the result as
\begin{equation*}
	S_n(2^{d+1}q) = 2^{nd}q^n + 2^d\frac{S_n(2^dq)}{2^{d-1}}
			 +2^{2d+1}\!\sum_{i=0}^{(n-2)/2}\!\binom{n}{2i} 2^{d(n-2i-2)}q^{n-2i} S_{2i}(2^dq).
\end{equation*}
By the induction hypothesis, the fraction is actually an odd integer. Since $nd > d$, we conclude that $v_2\!\left(S_n(2^{d+1}q)\right) = d$, as desired.

Similarly, if $n \ge 3$ is odd, then
\begin{equation*}
	S_n(2^{d+1}q)\!=\!2^{nd}q^n + 2^{2d}nq\frac{S_{n-1}(2^dq)}{2^{d-1}}
			+ 2^{3d+1}\!\!\sum_{i=0}^{(n-3)/2}\!\! \binom{n}{2i} 2^{d(n-2i-3)}q^{n-2i} S_{2i}(2^dq).
\end{equation*}
Again by induction, the fraction is an odd integer. Since $nd > 2d$, and $n$ and $q$ are odd, we see that $v_2\!\left(S_n(2^{d+1}q)\right) = 2d$, as required. This completes the proof of the lemma.
\end{proof}

\begin{proof}[Proof of Theorem~\ref{THM: Sigma(2^d)}]
When $m$ is even, write $m=2^dq$, where $d\ge1$ and $q$ is odd. Then $v_2(m(m+1)/2)=d-1$, and \eqref{EQ: m=2^dq} implies \eqref{EQ: main}.

If $m$ is odd, set $m+1=2^dq$, with $d\ge1$ and $q$ odd. Again we have $v_2(m(m+1)/2)=d-1$. From \eqref{EQ: m(m-1)/2} with $k=m$ we get $v_2(S_1(m))=d-1$, so that \eqref{EQ: main} holds for $n=1$. If $n>1$, then $nd > 2(d-1) \ge d-1$, and so \eqref{EQ: m=2^dq} and the relations$$S_n(m) = S_n(m+1)-(m+1)^n  \equiv S_n(m+1) \pmod{2^{nd}}$$ imply $v_2(S_n(m))=v_2(S_n(m+1))$ and, hence, \eqref{EQ: main}. This proves the theorem.
\end{proof}

\paragraph{Acknowledgments.} The authors are grateful to Tamas Lengyel and Pieter Moree for valuable comments and suggestions.

%%%  BIBLIOGRAPHY

{
\setlength{\parindent}{0cm}
\small
55 Lessard Avenue\par Toronto, Ontario, Canada M6S 1X6\par email: {\tt kieren@alumni.rice.edu}\par
\ \par
209 West 97th Street\par New York, NY 10025, USA\par email: {\tt jsondow@alumni.princeton.edu}\par }

\end{document}